\documentclass[12pt,reqno,tbtags]{amsart}

\usepackage{amsmath,amssymb,bbm,hyperref,url}

\usepackage{fullpage}

\numberwithin{equation}{section}

\newcommand{\Prob}{\mathbb{P}}

\newcommand{\R}{\mathbb{R}}
\newcommand{\Indi}[1]{\mathop{\mathbbm{1}_{#1}}\nolimits}

\newcommand\nIk{z_{k}}
\newcommand\nn{^{(n)}}
\newcommand{\dImax}{z_*} 

\newcommand{\bZ}{Z}
\newcommand{\bpT}{T}

\newcommand{\hp}{p}
\newcommand{\hm}{\mu} 
\newcommand{\bpC}{\hat c} 

\newtheorem{theorem}{Theorem}[section]
\newtheorem{lemma}[theorem]{Lemma}

\newtheorem{corollary}[theorem]{Corollary}

\theoremstyle{definition}

\newtheorem{remark}[theorem]{Remark}
\newtheorem*{acks}{Acknowledgements}

\theoremstyle{remark}

\newcommand{\refT}[1]{Theorem~\ref{#1}}
\newcommand{\refC}[1]{Corollary~\ref{#1}}
\newcommand{\refL}[1]{Lemma~\ref{#1}}
\newcommand{\refR}[1]{Remark~\ref{#1}}

\newcommand{\sumk}{\sum_{k=0}^\infty}

\newcommand\bigpar[1]{\bigl(#1\bigr)}
\newcommand\Bigpar[1]{\Bigl(#1\Bigr)}
\newcommand\biggpar[1]{\biggl(#1\biggr)}
\newcommand\lrpar[1]{\left(#1\right)}

\newcommand\parfrac[2]{\lrpar{\frac{#1}{#2}}}
\newcommand\punkt{.\spacefactor=1000}

\newcommand\ie{i.e\punkt}

\newcounter{CC}
\newcounter{cc}
\newcommand{\cc}{\stepcounter{cc}\ccx} 
\newcommand{\ccx}{c_{\arabic{cc}}}     
\newcommand{\ccdef}[1]{\xdef#1{\ccx}}     

\newcommand\Exp{\operatorname{Exp}}

\newcommand\gb{\beta}

\newcommand\gl{\lambda}

\newcommand\eps{\epsilon}

\renewcommand{\=}{:=}
\newcommand\dd{\,d}
\newcommand\rhs{right-hand side}

\begin{document}

\title{The extinction time of a subcritical branching process related to the SIR epidemic on a random graph}


\author{Peter Windridge}
\maketitle

\begin{abstract}

We give an exponential tail approximation for the extinction time of a 
subcritical multitype branching process arising from the SIR epidemic model on a random graph with given degrees,
where the type corresponds to the vertex degree.  
As a corollary we obtain a Gumbel limit law for the extinction time, when beginning with a large population.
 
Our contribution is to allow countably many types 
 (this corresponds to unbounded degrees in the random graph epidemic model, 
 as the number of vertices tends to infinity).  We only require
 a second moment for the offspring-type distribution featuring in our model.
\end{abstract}


\medskip{} Keywords: multitype branching process, exponential tail approximation, Gumbel, SIR epidemic.

\section{Introduction}\label{s:intro}

In this note we consider a continuous time Markov multitype branching process $\bZ = (\bZ_t(k); \; t \ge 0, \; k = 1,2,\ldots )$
arising from the 
 Susceptible-Infective-Recovered (SIR) epidemic model on a random graph with given degrees.  (We mention this connection
 only as motivation and do not explain it in detail.  If desired, the reader can consult \cite{JansonLuczakWindridgeSIR} 
 or \cite{BohmanPicollelli} for a construction in which the branching process
 studied here is apparent.  See also \refR{r:rgcm} below.)

Individuals in $\bZ$ are thought of as infective hosts carrying a number of spores.  
An individual's type $k \ge 1$ is simply the number of spores it has.  (We 
ignore individuals with no spores.)
Each spore, at a given rate $\gb > 0$, is released and 
gives rise to a new infective individual.  The new individual has a random type 
(\ie{} number of spores), $J$ say, 
chosen according to some given probability distribution, denoted $(\hp_j)_{j=0}^\infty$.  
(We allow $\hp_0 > 0$, and $J = 0$ means that no new infective is produced.)  
This leaves one individual of type $k-1$ (representing the original individual after losing a spore, assuming $k\ge 2$), 
and another with type $J$ (assuming $J \ge 1$). 
Furthermore, each individual (including its hosted spores) is removed from the population  
at rate $\rho \ge 0$, leaving nothing in its stead, regardless of type.  Thus an individual of type $k$
has an exponentially distributed lifetime of rate  $\rho + \gb k$, and is replaced by $0$, $1$ or $2$
individuals. 

This is a standard form of multitype branching process.  
Classical theory provides an exponential tail approximation
for the extinction time in the subcritical case, 
at least when there are finitely many types, \ie{} $(\hp_j)_{j=0}^\infty$ has finite support (see 
\cite{heinzmanextinctiontime}, and \cite[Chapter V, Theorem 11.1]{harris} for the single type result of Sewastjanow).  
In the general denumerable case, extinction itself is more delicate, 
for example extinction can be almost sure even though the expected population size tends to infinity \cite{hautphenne2013}.  
Tail approximations have been made in special cases, such as birth distributions of linear-fractional form \cite{Sagitov20132940}.  
However, there does not seem to be an applicable result for our rather simple model.
In this note we require only that $(\hp_j)_{j=0}^\infty$ has a second moment.

Denote by
\begin{equation}
q_k(t) = \Prob(\bZ_t \neq 0| \bZ_0(i) = \delta_{ki}, i\ge 1),
\end{equation}
the probability that the process survives till time $t \ge 0$ when it begins with a single 
host carrying $k \ge 1$ spores.  (We find it more intuitive to 
speak of hosts and their spore counts in the sequel, rather than using the branching process terminology of individuals and their type.)  The exponential approximation result 
we prove is as follows.

\begin{theorem}\label{t:qkexp}
Suppose $\hm \= \sumk k\hp_k > 0$,
\begin{equation}\label{e:decayrate}
\gl \= \rho + \gb(1 - \hm) > 0,
\end{equation}
and
\begin{equation}
\label{d:cubeUI}
\sumk k^2\hp_k < \infty.
\end{equation}

Then there exists a constant $\bpC \in (0,1]$ such that,
for any $a < \min\{\lambda, \beta\}$,
\begin{equation}\label{e:qexp}
q_k(t) = \bpC k e^{-\gl t}( 1 + O(k e^{-a t})) ~\mathrm{as}~ t \to \infty
\end{equation}
for all $k \ge 1$.  
\end{theorem}

The condition in \eqref{e:decayrate} means that $\bZ$ is subcritical.  
In the ommitted case of $\hm = 0$,
no new hosts occur and $q_k(t) = e^{-\rho t}(1 - (1 - e^{-\beta t})^k)$ 
 is just the probability that the initial host is still present and at least one of its spores remains.  

\begin{remark}\label{r:c}
Any value of $\bpC \in (0,1]$ is possible.  Indeed, if $\rho = 0$ then all spores behave independently
and form a single type branching process.  The 
case $\hp_0 + \hp_1 + \hp_2 = 1$ is linear fractional \cite[III.5, p. 109]{AthreyaNey} 
and $q_1(t)$ can be computed explicitly.  
More specifically, $\hp_0 + \hp_2 = 1$ corresponds 
to a linear birth and death chain and for $\hp_0 \neq \hp_2$ we have
\begin{equation}
q_1(t) = \frac{(\hp_0- \hp_2)e^{-\gb(\hp_0 - \hp_2)t} }{\hp_0 - \hp_2 e^{-\gb(\hp_0 - \hp_2)t}}.
\end{equation}
Thus, the leading constant in \refT{t:qkexp} is $\bpC = 1 - \hp_2/\hp_0$, for any $\hp_2 <\hp_0$.
\end{remark}

\refT{t:qkexp} can be used to examine the distribution of the duration of a
subcritical epidemic where there are initially large number of infective hosts.  
The utility of such a result is perhaps not immediately apparent, so we remark
that the situation described also arises at the end of a major epidemic, 
when an outbreak has become so large that it starts shrinking 
due to there being few remaining susceptible individuals.

\begin{corollary}\label{c:doubleexp}
Adopt the setting of \refT{t:qkexp}.   Suppose that 
for each $n \ge 1$ we have a sequence $z_k = z_k\nn$, $k \ge 1$ of natural numbers
such that $\sumk k\nIk \to \infty$, 
and 
\begin{equation}\label{d:mystery0}
\sumk k^2\nIk = o\lrpar{\sumk k\nIk}^{1 + (a/\gl)}
\end{equation}
as $n\to\infty$, for some $a > 0$ satisfying the condition in \refT{t:qkexp}. 

Suppose $\bZ_0(k) = z_k$ for every $k$ and $\bpT \= \inf\{ t \ge 0: \bZ_t = 0\}$.  Then, for any fixed $w \in \R$,
\begin{equation}
\label{e:doubleexp}
 \Prob\biggpar{\gl \bpT \le \ln\Bigpar{\bpC \sumk k \nIk} + w} \to e^{-e^{-w}}
\end{equation}
as $n\to\infty$.
\end{corollary}

The double exponential distribution on the \rhs{} of \eqref{e:doubleexp} is known as the Gumbel distribution.
It arises from taking the maximum of a large number of independent random variables with exponential tails,
and thus is common in the context of branching process extinction times \cite{pakes89, Jagers10042007, heinzmanextinctiontime}.

\begin{remark}\label{r:rgcm}
When studying the SIR epidemic on a random graph with given degree sequence, one typically
constructs (or `reveals') relevant parts of the graph 
while the disease spreads, via a device known as the configuration model 
(see \cite{JansonLuczakWindridgeSIR} and the references therein). 
For the benefit of readers familiar with the configuration model approach, 
it is worth noting that the `spores' in this paper correspond to half-edges, and a `host' is just a vertex in the graph.
To apply our result, the probability distribution $(\hp_k)_{k=0}^\infty$ should be a size-biased 
transform of the graph degree distribution.  In particular, our second moment condition translates to a third moment requirement
for the vertex degree distribution.  The details of a pathwise coupling, and conditions needed for it to hold with high probability,
are left for future work.
\end{remark}

The proofs of \refT{t:qkexp} and \refC{c:doubleexp}, in sections \ref{s:qkexp}
and \ref{s:doubleexp} respectively, occupy the remainder of this note.


\acks
Funded by EPSRC grant EP/J004022/2 (principal investigator Malwina Luczak).  The author thanks
Svante Janson and Malwina Luczak for useful comments.

\section{Proof of \refT{t:qkexp}}\label{s:qkexp}

The general idea is to show that $q_k(t) \sim k q_1(t)$, 
as $t \to \infty$, by controlling the dependencies between spores of the same host, see \refL{l:qkadd} below.  
First we need preliminary bounds 
on $q_1(t)$.  Let us fix $a < \min\{\gl,\gb\}$.  
Suppose initially there is a single host with one spore.  
Let $R$ denote its $\Exp(\rho)$ removal time, 
and $F$ the $\Exp(\gb)$ release time of its single spore.  The process survives till a given time $t > 0$
if and only if either the spore is released before $t$ (and necessarily before the host is removed) and its 
progeny persist till time $t$, or neither the spore is released nor the host removed by $t$.  Thus
\begin{align}
q_1(t) & = \int_0^t \sumk \hp_k q_k(t-f)\Prob(F \in \dd f) \Prob(R > f) + \Prob(F,R > t) \nonumber \\
& = \int_0^t \sumk \hp_k q_k(t-f)\beta e^{-(\rho+\gb)f} \dd f + e^{-(\rho+\gb)t}, \label{e:q1int}
\end{align}
from which we obtain the differential equation
\begin{equation}\label{e:q1ode}
 q_1'(t) = -(\rho + \gb)q_1(t) + \gb \sumk \hp_k q_k(t), \;\; t \ge 0, \; q_1(0)=1.
\end{equation}
This also follows from the Kolmogorov backwards equations, but deriving it
from the integral is a useful warmup for the calculations below.

Now suppose there is initially a host with $k \ge 2$ spores.  Survival 
of the process till time $t$ 
implies that at least one of the $k$ initial spores, or its progeny, persist till $t$.  It follows 
that $q_k(t) \le k q_1(t)$.  Using this inequality with \eqref{e:q1ode} yields
\begin{equation}
q_1'(t) \le -(\rho + \gb)q_1(t) + \gb\sumk k \hp_k q_1(t) = -\gl q_1(t),
\end{equation}
and so $e^{\gl t} q_1(t)$ is non-increasing in $t$.  It is positive and so the limit
\begin{equation}\label{e:bpC}
\bpC \= \lim_{t\to\infty} e^{\gl t} q_1(t)
\end{equation}
exists.  We have $\bpC \le 1$ by monotonicity of $e^{\gl t} q_1(t)$ and the fact that $q_1(0) = 1$.

It will later transpire that $\bpC > 0$.  For now we will lower bound $q_1(t)$ 
by truncating the spore count distribution $(p_k)_{k = 0}^\infty$.  More precisely, 
take any positive $\eps < \min\{\gl,\gb\} - a$, and
choose $k_0 \ge 1$ large enough that $\sum_{ k =1}^{k_0} k p_k > \mu - \eps/\beta$.
Our branching process $\bZ$ dominates a modified process in which spore counts of new hosts are 
distributed as $J\cdot \Indi{ J \le k_0}$, where $J \sim (p_k)_{k = 0}^\infty$.
This modified process has finitely many types, and so the exponential tail 
approximation of \cite{heinzmanextinctiontime} applies,
with decay rate $\rho + \gb  - \gb \sum_{ k =1}^{k_0} k p_k < \gl + \eps$.  In particular 
there exists $\cc\ccdef\ccqL = \ccqL(\eps)  > 0$ such that
\begin{equation}\label{e:q1crude}
q_1(t) \ge \ccqL e^{-(\gl + \eps)t}
\end{equation}
for all $t \ge 0$.


\begin{lemma}\label{l:qkadd}
For all $k \ge 1$ we have
\begin{equation}\label{e:qkbon0}
 q_k(t) = kq_1(t)(1 + O(k e^{-at})),
\end{equation}
as $t \to \infty$.
\end{lemma}

\begin{proof}
As already mentioned we have $q_k(t)\le kq_1(t)$.  
For the lower bound, suppose there is initially a single host with $k \ge 2$ spores.
Let $T_i$, $i = 1,\ldots,k$
denote the total time that
spore $i$, or its progeny, persist for.  Thus
\begin{equation}
\label{e:qkTi}
q_k(t) = \Prob\left(\bigcup_{i = 1}^k \{ T_i > t \} \right).
\end{equation}
The Bonferroni inequality yields 
\begin{equation}\label{e:qkbon}
q_k(t) \ge k q_1(t) - k^2 \Prob(T_1, T_2 > t),
\end{equation}
where we used $\Prob(T_1  > t) = q_1(t)$, and the fact that any pair of times have the same joint distribution.
If $\rho = 0$ then $T_1$ and $T_2$ are independent and there is nothing to prove.  
In the sequel we assume $\rho > 0$ 
and control the dependency between the $T_i$ using the fact that the progeny of different spores behave independently.


Suppose $F_1$ and $F_2$ denote the
independent  $\Exp(\gb)$ release times of spores $1$ and $2$, and $R$ denotes the $\Exp(\rho)$ removal time of the host.  Then, ignoring the null 
events $\{ R < t$ and $T_i, F_i > t\}$, we have
\begin{align}
\Prob(T_1, T_2 > t) & = \Prob(R,F_1,F_2 > t) + 2\Prob(T_1, F_2, R >t, ~\mathrm{and}~ F_1 < t) + \nonumber \\
& \qquad  + \Prob(T_1,T_2 > t, ~\mathrm{and}~  F_1,F_2 < t). \label{e:qx}
\end{align}
The first probability on the \rhs{} is simply $e^{-(\rho + 2\gb)t}$.  
The second probability 
can be bounded by writing it as an integral against the spore release time, and then using $q_k(t) \le k q_1(t)$, $\hm = \sumk k\hp_k$ and $q_1(t) \le e^{-\gl t}$ as follows. 
\begin{align}
\Prob(T_1, F_2, R >t, ~\mathrm{and}~ F_1 < t) & = e^{-(\rho + \gb)t} \int_0^t \sumk \hp_k q_k(t-f) \Prob(F_1 \in \dd f) \nonumber \\
& \le \hm  e^{-(\rho + \gb)t} \int_0^t q_1(t-f)\gb e^{-\gb f} \dd f \nonumber\\
& \le \hm \gb e^{-(\rho + \gb + \gl)t} \int_0^t e^{(\gl - \gb)f}\dd f \nonumber\\
& = \hm \gb  e^{-(\rho + \gb + \gl)t} \parfrac{e^{(\gl -\gb)t} - 1}{\gl - \gb},\label{e:qx2}
\end{align}
assuming $\gl \neq \gb$, otherwise the last integral evaluates to $t$.
In both cases,
\begin{align}
\Prob(T_1, F_2, R >t, ~\mathrm{and}~ F_1 < t) & = O\bigpar{ e^{-(\rho + 2\gb)t} + (1 + t)e^{-(\rho + \gb + \gl)t} } \nonumber \\
& = O\bigpar{ e^{-(\rho + 2\gb)t} + e^{-(\gb + \gl)t} }\label{e:qx20}.
\end{align}

Finally we turn to the third probability on the \rhs{} of \eqref{e:qx}.  Consider
the probability $g(t,r)$ that a given spore is released before time $r > 0$ (assuming the host is not removed first) and 
has progeny who persist till time $t \ge r$.  Repeating calculations 
similar to those in \eqref{e:qx2} shows it satisfies
\begin{equation}
g(t,r) = \int_0^r \sumk \hp_k q_k(t-f) \Prob(F_1 \in \dd f)  \le 
{\hm \gb } e^{-\gl t} \parfrac{ e^{(\gl - \gb)r} - 1}{\gl -\gb}
\end{equation}
for $\gl \neq \gb$, otherwise the bracketted term is $r$.  
Thus, the probability that two given spores are released before time $r$ (ignoring removal of the host) and bear progeny persisting till time $t \ge r$ 
satisfies
\begin{equation}\label{e:gtr2}
g(t,r)^2 \le \cc\ccdef\ccqx e^{- 2\gl t} \bigpar{ e^{2(\gl - \gb)r} + 1 + r^2},
\end{equation}
for some constant $\ccqx > 0$.  Integrating \eqref{e:gtr2} against the removal time density
yields
\begin{align}
\Prob(T_1,T_2 > t, ~\mathrm{and}~  F_1,F_2 < t) & = \int_0^t g(t,r)^2\Prob(R \in dr) + g(t,t)^2\Prob(R > t) \nonumber \\
& = O\bigpar{ e^{-2\gl t}  + e^{-(\rho + 2\gb)t} + (1 +t^2) e^{-(2\gl + \rho)t} } \nonumber \\ 
& = O\bigpar{ e^{-2\gl t}  + e^{-(\rho + 2\gb)t} }. \label{e:qx3}
\end{align}

Armed with these estimates we return to  \eqref{e:qx} and find
\begin{equation}
\Prob(T_1, T_2 > t) = O\bigpar{ e^{-(\rho+2\gb)t} + e^{-(\gb + \gl)t} + e^{-2\gl t}},
\end{equation}
and using \eqref{e:q1crude} we have
\begin{equation}\label{e:fcq}
\Prob(T_1, T_2 > t)/q_1(t) = O\bigpar{ e^{-(\rho+2\gb -\gl - \eps)t} + e^{-(\gb - \eps)t} + e^{-(\gl - \eps)t}}.
\end{equation}
Now, $\rho + 2\gb - \gl = \gb(1 + \mu) > \gb$.  So, 
each negative exponent in \eqref{e:fcq} is at least $a$ by our choice of $\eps$.  
The desired relationship \eqref{e:qkbon0} now follows from the Bonferroni inequality \eqref{e:qkbon}.
\qed \end{proof}

We will now prove that $\bpC > 0$, \ie{} that $q_1(t)$ really 
does decay like $e^{-\lambda t}$.  Apply \refL{l:qkadd} to get 
\begin{align}
\sumk \hp_k q_k(t) & \ge \sumk k\hp_k q_1(t) - O\Bigpar{ q_1(t) e^{-a t} \sumk k^2\hp_k} \nonumber \\ 
& = q_1(t) ( \hm + O(e^{-a t}) ),\label{e:qMA}
\end{align}
using the assumption \ref{d:cubeUI} that $\sumk k^2\hp_k < \infty$.  Combining 
this with the differential equation \eqref{e:q1ode} for $q_1(t)$ yields
\begin{align}
\frac{d}{dt} \bigpar{\ln q_1(t) + \gl t} & = \frac{q_1'(t)}{q_1(t)} + \gl \nonumber \\ 
& = -( \rho + \gb) + \frac{\gb}{q_1(t)} \sumk  \hp_k q_k(t) + \gl \nonumber \\
& \ge -\cc\ccdef\ccqf e^{-at},
\label{e:lnq1ode}
\end{align}
for some constant $\ccqf > 0$.  Integrating both sides,
\begin{equation}
\ln q_1(t) +\gl t \ge -\ccqf \int_0^t e^{-as}\dd s \ge -\ccqf/a. \nonumber 
\end{equation}
Hence $\bpC = \lim_{t\to\infty} e^{\lambda t}q_1(t) \ge \exp(-\ccqf/a) > 0$. 

Finally, $e^{\gl t}q_1(t)$ is non-increasing so $q_1(t) \ge \bpC e^{-\gl t}$.  Moreover, 
using \eqref{e:lnq1ode} again,
\begin{equation}
\ln (e^{\gl t} q_1(t)) \le \ccqf\int_t^\infty e^{-as}\dd s + \lim_{s\to\infty}  \ln (e^{\gl s} q_1(s) ) = O(e^{-at}) + \ln\bpC.
\end{equation}
But $\exp(O(e^{-at})) = 1 + O(e^{-at})$ as $t\to\infty$ and so \eqref{e:qexp} holds for $k = 1$.  The 
result for $k \ge 2$ follows from \eqref{e:qkbon0}.
\qed

\begin{remark}
 The proof in \cite{heinzmanextinctiontime} (for finitely many types) 
uses a linear approximation to \eqref{e:q1ode} its analogues for $q_k$, $k\ge 2$.  A 
general result about perturbations of linear equations is then 
applied to derive \eqref{e:qexp}.  We took a different approach here
due to not finding a suitable perturbation result for infinite 
systems of differential equations.
\end{remark}

\section{Proof of \refC{c:doubleexp}}\label{s:doubleexp}

Take $t = t(n) \=  \gl^{-1}\left(\ln\bigpar{\bpC \sumk k \nIk} + w\right)$, so $t > 0$ for $n$ large enough.   
Using the branching property, \refT{t:qkexp}, 
and finally assumption \eqref{d:mystery0},
\begin{align*}
\ln \Prob(\bpT \le t) & = \ln \prod_{k = 0}^{\infty} (1 - q_k(t))^{\nIk} = \sumk  \nIk \ln (1 - q_k(t)) \\
& \le -\sumk \nIk q_k(t) \\
& \le -\bpC e^{-\gl t} \sumk k\nIk + O\bigpar{e^{-(a + \gl)t}\sumk k^2 \nIk} \\
& = -e^{-w} + O\parfrac{\sumk k^2\nIk}{\lrpar{\sumk k \nIk}^{1 + (a/\gl)}} \to -e^{-w},
\end{align*}
as $n\to\infty$.

On the other hand, assumption \eqref{d:mystery0} implies the maximum number of spores 
$\dImax \= \max\{ k : \nIk \ge 1 \} $ of any initial host satisfies $\dImax = o(\sumk k\nIk)$.  So, 
for any $k\le \dImax$ we have
\begin{equation}
q_k(t) \le kq_1(t) \le ke^{-\gl t} \le \frac{\dImax}{e^w \bpC \sumk k\nIk} \to 0,
\end{equation}
as $n\to\infty$.  This implies $q_k(t) < 1/2$ eventually (for any $k$ with $\nIk \ge 1$), and consequently
\begin{equation}
\ln (1 - q_k(t)) \ge - q_k(t)(1 + q_k(t)) \ge -q_k(t) - k^2 e^{-2\gl t},
\end{equation}
using the inequality $\ln (1 - h) \ge -h(1+h)$ for $h \le 1/2$.
It follows that
\begin{equation}
\ln\Prob(\bpT \le t) = \sumk \nIk \ln (1 - q_k(t)) \ge -\sumk \nIk q_k(t) - \frac{\sumk k^2\nIk}{(e^w \bpC \sumk k\nIk)^2} \to -e^{-w},
\end{equation}
as $n\to\infty$, as in the upper bound.

\qed

\bibliography{sircmbpexptail}
\bibliographystyle{apt}

\url{pete@windridge.org.uk}

\end{document}